\documentclass[letterpaper, 10 pt, conference]{ieeeconf}  

\IEEEoverridecommandlockouts                              

\usepackage{graphics} 
\usepackage{epsfig} 
\usepackage{mathptmx} 
\usepackage{times} 
\usepackage{amsmath} 
\usepackage{amssymb}  
\usepackage{subfigure}
\usepackage{color}
\usepackage{flushend}
\usepackage{eurosym}
\usepackage{amsmath,theorem}
\usepackage{amssymb}
\usepackage{amsfonts}
\usepackage{graphics}
\usepackage{psfrag}
\usepackage{array}
\usepackage{shortvrb}
\usepackage{epsf}
\usepackage{graphicx}
\usepackage{rotating}
\usepackage{cite}
\usepackage{wasysym}
\usepackage[usenames]{xcolor}

\usepackage{graphicx}
\usepackage{eurosym}
\usepackage{amssymb}
\usepackage{amsmath,mathtools}
\usepackage{amsfonts}
\usepackage{epstopdf}
\usepackage{epsf,subfigure}
\usepackage{psfrag}
\usepackage{graphics}
\usepackage{color} 

\newtheorem{definition}{Definition}
\newtheorem{theorem}{Theorem}

\newtheorem{lemma}{Lemma}

\newcommand{\qed}{\nobreak \ifvmode \relax \else
      \ifdim\lastskip<1.5em \hskip-\lastskip
      \hskip1.5em plus0em minus0.5em \fi \nobreak
      \vrule height0.75em width0.5em depth0.25em\fi}

\newcommand{\rem}[1]{}

\title{\LARGE \bf
Stability and Control of Power Systems using Vector Lyapunov Functions and Sum-of-Squares Methods*
}

\author{Soumya Kundu$^{1}$ and Marian Anghel$^{2}$
\thanks{*This work was supported by the U.S. Department of Energy
through the LANL/LDRD Program.}
\thanks{$^{1}$Soumya Kundu is with the Center for Nonlinear Studies and Information Sciences Group (CCS-3), Los Alamos National Laboratory, Los Alamos, USA
        {\tt\small soumya@lanl.gov}}%
\thanks{$^{2}$Marian Anghel is with the Information Sciences Group (CCS-3), Los Alamos National Laboratory, Los Alamos, USA
        {\tt\small manghel@lanl.gov}}%
}

\begin{document}

\maketitle
\thispagestyle{empty}
\pagestyle{empty}

\begin{abstract}

 Recently{,} sum-of-squares (SOS) based methods have been used for the stability analysis and control synthesis of polynomial dynamical systems. 
 This analysis framework was also extended to non-polynomial dynamical systems, including power systems, using an algebraic 
 reformulation technique that recasts the system's dynamics into a set of polynomial differential algebraic equations. 
 Nevertheless, for  large scale dynamical systems this method becomes inapplicable due to its computational complexity.  
 For this reason we develop a subsystem based stability analysis approach using vector Lyapunov functions and introduce
  a parallel and scalable algorithm to infer the stability of the interconnected system with the help of the subsystem Lyapunov functions. 
  Furthermore, we design adaptive and  distributed  control laws that guarantee  asymptotic stability under a given external disturbance. 
  Finally,  we apply  this algorithm  for the stability analysis and control synthesis of a network preserving power system.

\end{abstract}

\section{INTRODUCTION}


Recently, a methodology for the algorithmic construction
of Lyapunov functions for the transient stability analysis
of classical power system models was introduced \cite{Anghel:2013}. 
The proposed methodology
uses advances in the theory of positive polynomials,
semidefinite programming, and sum of squares decomposition,
that have provided powerful nonlinear tools for the analysis of
systems with polynomial vector fields \cite{Parrilo:2000, Antonis:2002,Antonis:2005, 
Antonis:2005b, Wloszek:2005, Wloszek:2003}. In order to apply these
techniques to power systems described by trigonometric
nonlinearities  an algebraic reformulation technique to
recast the system's dynamics into a set of polynomial differential
algebraic equations was used \cite{Antonis:2005, Anghel:2013}.
However, the sum-of-squares (SOS) approach does not scale well and, consequently, SOS methods only work
 for small systems with only a few state variables. For a large scale power system, it is then important to devise a subsystem based stability analysis approach using the concept of vector Lyapunov functions and the decomposition-aggregation method 
 \cite{Siljak:1979,Weissenberger:1973}. 
 \rem{Generally from the subsystem Lyapunov functions, a global Lyapunov function (or equivalent) candidate is constructed whose properties then determine the stability of the complete system. }
%

Formulations using vector Lyapunov functions 
\cite{Bellman:1962,Bailey:1966} are computationally attractive because of their parallel structure and scalability. In \cite{Weissenberger:1973}, it was shown that if the subsystem Lyapunov functions and the interactions satisfy certain conditions, then application of comparison equations 
can provide a certificate of exponential stability of the interconnected systems. The approach  we propose here uses
instead a generalization to interconnected systems of SOS based robust stability analysis 
techniques developed to estimate the  effect of parametric uncertainty~\cite{Antonis:2005b} and external disturbances 
on a polynomial system~\cite{Wloszek:2003}.

In this work  we use  sum-of-squares analysis  methods to devise a new algorithmic certification of asymptotic stability via the vector Lyapunov function approach. While this approach is generic, we apply the proposed algorithm to analyze the stability of
a structure preserving power system model \cite{Bergen:1981,Hill:1982}. 
 %
 %
  Unlike the network preserving models studied in the literature
 we do allow for nonzero transfer conductances in the transmission lines. 
 The network is decomposed into a number of low order interacting subsystems.   For each such subsystem, a SOS based expanding interior algorithm \cite{Wloszek:2003,Anghel:2013} is used to obtain estimate of region of attraction as sub-level sets of polynomial Lyapunov functions. Finally a sum-of-squares based scalable and parallel algorithm is used to certify stability in the sense of Lyapunov of the interconnected system by using the subsystem Lyapunov functions computed in the previous step. A distributed control strategy is proposed that can guarantee asymptotic stability of the interconnected system under given disturbances. 
 
 Following some brief background in Sec.~\ref{S:background} we outline the problem statement in Sec.~\ref{S:problem}. An algorithmic approach to certifying asymptotic stability is presented in Sec.~\ref{S:stability} while a distributed control strategy is discussed in Sec.~\ref{S:control}. Sec.~\ref{S:results} shows an application of our stability analysis and control approach to a network of three generators and six frequency dependent loads. We conclude the article in Sec.~\ref{S:conclusion}.

\section{\MakeUppercase{Basic Concepts and Background}}
\label{S:background}
\rem{Before formulating the problem, let us briefly review some of the key concepts behind our analysis.} We will first discuss how the stability of a dynamical system can be analyzed by constructing suitable Lyapunov functions. Then we briefly refer to sum-of-square polynomials and a very useful result which helps us in formulating the sum-of-squares problems. 

\subsection{Lyapunov Stability Methods}
\label{S:Lyap}
Let us consider the dynamical system described by the following polynomial differential algebraic equations (DAE)
\begin{subequations}\label{E:dae}
\begin{align}
\dot{z} &= F(z)\, , \\
0 &= G(z)\, ,
\end{align}
\end{subequations}
where $z \in  \mathbb{R}^m$, and $ F :   \mathbb{R}^m \rightarrow  \mathbb{R}^m$, $ G :  \mathbb{R}^{m} \rightarrow \mathbb{R}^{q}$ are vectors of polynomial functions.
We
assume without loss of generality that the origin is a stable equilibrium point for this
system, i.e.  $F(0) = 0$ and $G(0) = 0$.

The following extension of  Lyapunov stability theorem  to 
dynamical systems described by  differential algebraic equations presents a 
sufficient condition of stability through the construction of a certain positive definite function \cite{Antonis:2005b,Anghel:2013}.
\begin{theorem}\label{T:Lyap}
 If there exists an open set $D \subset  \mathbb{R}^m$,  with  $D$   a semi-algebraic domain 
 defined by the following inequality and equality constraints,
\begin{equation}\label{E:D_def}
  D = \lbrace z \in  \mathbb{R}^m \mid \beta - p(z) \geq 0, G(z) = 0 \rbrace\, ,
\end{equation}
with $p(z)$ a positive definite polynomial and $\beta  > 0$ to ensure that $D$ is connected and contains 
 $ z = 0 $,  and a continuously differentiable function $V: D \rightarrow \mathbb{R}$ such that $V(0) = 0$, and
 \begin{subequations}\label{E:Lyap}
\begin{align}
 V(z) & > 0\,  ,  \forall z \in  D \diagdown \lbrace 0 \rbrace\, ,\\
  - \dot{V}(z)  & > 0\, ,  \forall z \in    D \diagdown \lbrace 0 \rbrace\, ,
\end{align}
\end{subequations}
where $ \dot{V}(z) = \nabla{V}^T \cdot F(z)$, then $z=0$ is an asymptotically stable equilibrium of \eqref{E:dae}.
\end{theorem}
When there exists such a function $V(z)$, the region of attraction (ROA) of the stable 
equilibrium point at origin can be (conservatively) estimated as
\begin{subequations}\label{E:ROA}
\begin{align}
\mathcal{R}_A&:=\left\lbrace z \in D\left| V(z)\leq \gamma^{max}\right.\right\rbrace\\
\text{where,}~\gamma^{max}&:=\arg\max_\gamma\left\lbrace z \in D  \left| V(z)\leq\gamma\right.\right\rbrace \subseteq D
\end{align}
\end{subequations}
Without any loss of generality, the Lyapunov function can be scaled by $\gamma^{max}$, so that the ROA is given by,
\begin{equation}\label{E:ROA2}
\mathcal{R}_A:= \left\lbrace z \in D \left| {V}(z)\leq 1\right.\right\rbrace
\end{equation}
Henceforth, we would assume that the ROA is estimated to be the sub-level set of ${V}(z)=1$.

\subsection{Sum-of-Squares and Positivestellensatz Theorem}
\label{S:SOSmethod}
 While Thm.~\ref{T:Lyap} gives a sufficient condition for stability, it is  not a trivial task to find a \rem{suitable} function $V(z)$ that satisfies the conditions of stability. Relatively recent studies have explored how sum-of-squares (SOS) based optimization techniques can be utilized in finding Lyapunov functions by restricting the search space to SOS polynomials \cite{Wloszek:2003,Parrilo:2000,Tan:2006,Anghel:2013}. Let us denote $\mathcal{R}_m$ as the set of all polynomials in $z \in \mathbb{R}^m$. Then,
\begin{definition}
A multivariate polynomial $p(z) \in \mathcal{R}_m$ is a sum-of-squares (SOS) if there exist some polynomial functions $H_i(z), i = 1\ldots r$ such that 
$p(z) = \sum_{i=1}^r H_i^2(z)$,
and the set of all such SOS polynomials is denoted by
\begin{align}
\Sigma_{m} &:= \left\lbrace p(z)\in\mathcal{R}_m\left| ~p \text{ is SOS}\right.\right\rbrace \, .
\end{align}
\end{definition}
Given a polynomial $p\in\mathcal{R}_m$, checking if it is SOS is a semi-definite problem which can be solved with a MATLAB$^\text{\textregistered}$ toolbox SOSTOOLS \cite{sostools13,Antonis:2005a} along with a semidefinite programming solver such as SeDuMi \cite{Sturm:1999}.

Hence, if we search for a polynomial Lyapunov function in Thm.~\ref{T:Lyap},  and if 
we relax the polynomial non-negativity
conditions to appropriate polynomial sum of squares (SOS)
conditions, testing SOS conditions can then be done efficiently
using semidefinite programming (SDP) \cite{Parrilo:2000}.
Moreover,  an important result from algebraic geometry, called Putinar's Positivstellensatz theorem
\cite{Putinar:1993,Lasserre:2009}, helps in translating the SOS conditions into SOS feasibility problems. Before stating the theorem, let us define:
\begin{definition}\label{D:quad mod}
Given $g_j\in\mathcal{R}_m$, for $j=1,2,\dots,r$, the quadratic module generated by $g_j$'s is 
$\mathcal{M}(g_1,g_2,\dots,g_r):=\left\lbrace \sigma_0 + \sum_{j=1}^r \sigma_j g_j \left \vert \sigma_0,\sigma_j\in\Sigma_m,\forall j\right.\right\rbrace$\, .
\end{definition}
Then the Putinar's Positivestellensatz theorem states~\cite{Lasserre:2009},
{\begin{theorem}\label{T:Putinar}
Let $\mathcal{K}= \left\lbrace z \in\mathbb{R}^n\left\vert g_1(z) \geq 0, \dots , g_r(z)\geq 0\right.\right\rbrace$ be a
compact set. Suppose there exists $u(z)\in\mathcal{R}_m$ such that 
\begin{subequations}\label{E:Putinar}
\begin{align}
& u(z)\in\mathcal{M}(g_1, g_2,\dots , g_r),\\
\text{and,}~&\left\lbrace z \in\mathbb{R}^n\left\vert u(z)\geq 0\right.\right\rbrace~\text{is compact.}
\end{align}
\end{subequations}
If $p(z)$ is positive on $\mathcal{K}$, then $p(z)\in \mathcal{M}(g_1, g_2,\dots , g_r)$.
\end{theorem}}
It can be shown that for all  $g_i$'s used in this work the constraints (\ref{E:Putinar}) {would be redundant, i.e. the existence of $u(z)$ would be guaranteed.
If we further note that the equality constraints defined by the  components of $G(z)$ 
can be expressed by the pairs of inequalities $g_k \geq 0, \,g_k \leq 0,\, k=1,\ldots,q$, 
and if we define $g_{q+1} = \beta - p$, then using Thm.~\ref{T:Putinar}, with $r=q+1$,
 the search for $V(z)$ becomes a search for a feasible solution of the following problem 
 with SOS constraints~\cite{Wloszek:2003, Anghel:2013}:
\begin{subequations}\label{E:sos_lyap}
\begin{align}
 V  - \lambda_{1}^T G - s_1 (\beta - p) - \phi_1 &\in   \Sigma_m  \\
 -  \dot{V}  - \lambda_{2}^T G  - s_2 (\beta - p) - \phi_2 &\in   \Sigma_m
\end{align}
\end{subequations}
where $ s_i \in   \Sigma_m$, 
$ \lambda_i \in \mathcal{R}_{m}^{q}, $ are vectors of polynomial functions, and 
we choose $ \phi_i(z) = \epsilon_i \sum_{k=1}^m z_k^2$, $\epsilon_i > 0$, $i=1,2$.

Similarly, the search for ${\gamma^{max}}$ in Eq.~\eqref{E:ROA} can be formulated as a SOS programming problem 
and solved using a bisection search over $\gamma$ --- see \cite{Wloszek:2003, Anghel:2013} for details.

\section{\MakeUppercase{Problem Outline}}
\label{S:problem}

 Given the full dynamical system (\ref{E:dae}), we seek to decompose it into 
 $S$ weakly interacting subsystems as
\begin{subequations}\label{E:dae_decomp}
\begin{align}
&\dot{z}_i = F_i(z_i) + H_i(z), \quad  i =1,2,\dots,S, \\
&0 = G_i(z_i)\, , \\
&F_i(0)=0,~H_i(0)=0
\end{align}
\end{subequations}
where $z_i \in  \mathbb{R}^{m_i}$ is the state of the $i$-th subsystem $S_i$,
 $F_i$'s denote the isolated subsystem dynamics,  
$ G_i \in \mathcal{R}_{m_i}^{p_i}$ are vectors of polynomials
defining the algebraic constraints of the subsystem, 
and $H_i$'s are the interactions from the neighbors \cite{Jocic:1978}. 
If we denote by $\mathcal{Z}_i$ the subset of state points that correspond to $z_i$ then 
their union, i.e.
\begin{equation}
\mathcal{Z} = \mathcal{Z}_1 \cup \mathcal{Z}_2 \cup \ldots \cup \mathcal{Z}_s\, ,
\end{equation}
where $\mathcal{Z} = \lbrace  z_1, z_2, \ldots, z_m \rbrace$, represents the set of state points of 
the whole dynamical system \eqref{E:dae}. In the decomposition used in this paper 
the subsets $\mathcal{Z}_i, i=1,\ldots,S,$ are not disjoint.

We  assume that the interactions can be expressed as,
\begin{align}
\forall i, ~H_i(z) = \sum_{j\neq i} H_{ij}(z_i, z_j)
\end{align}
where $H_{ij}$ quantifies how the states $z_j$ of subsystem $j$ affect the dynamics of $z_i$. Let us also denote by
\begin{align}\label{E:Ni}
\mathcal{N}_i := \left\lbrace i\right\rbrace\cup\left\lbrace j\left\vert ~\exists \, z_i,z_j, ~\text{s.t.}~H_{ij}\left(z_i,z_j\right)\neq 0 \right.\right\rbrace,
\end{align}
the set of neighbors of node $i$ (including the subsystem $i$ itself). We assume that the isolated subsystems are individually (locally) stable, and  there exist Lyapunov functions for each of the isolated subsystems.
The goal is to develop a framework for the stability analysis of the full interconnected system by using the local subsystem Lyapunov functions and considering the neighbor interactions.  
In the work reported here we use an "extreme" decomposition in which the subsystems are defined by the 
nodes of the original network together with a reference node that is shared by all subsystems. A similar approach has been used in \cite{Jocic:1977}. A decomposition algorithm  proposed in \cite{Antonis:2012}, and used for a classical power system model in \cite{Anghel:2013b},  was not used  here, but will be used in the future to analyze the impact the decomposition has on the performance of the algorithm.

Given a decomposition \eqref{E:dae_decomp}, our next goal is to find polynomial 
Lyapunov functions $V_i(z_i) \in \mathcal{R}_{{m}_i} $ for each isolated subsystem
$  i =1,2,\dots,S,$
\begin{subequations}\label{E:ss_dae}
\begin{align}
&\dot{z}_i = F_i(z_i) ,  \\
&0 = G_i(z_i)\, .
\end{align}
\end{subequations}
The first step in this search is to solve  the SOS program \eqref{E:sos_lyap} which for each subsystem $\mathcal{S}_i$ is
 formulated as 
\begin{subequations}\label{E:ss_initV}
\begin{align}
 V_i  - \lambda_{i1}^T G_i - s_{i1} (\beta_i - p_i) - \phi_{i1} &\in   \Sigma_{m_i}  \\
 -  \dot{V}_i  - \lambda_{i2}^T G_i  - s_{i2} (\beta_i - p_i) - \phi_{i2} &\in   \Sigma_{m_i}
\end{align}
\end{subequations}
where $\beta_i > 0$, $ s_{i1}, s_{i2} \in   \Sigma_{m_i} $, 
$ \lambda_{i1}, \lambda_{i2} \in  \mathcal{R}_{m_i} ^{p_i}, $  
 and $\phi_{i1}(z_i),~\phi_{i2}(z_i),~p_i(z_i)$ are positive definite polynomials. {Starting from an initial Lyapunov function candidate obtained by solving \eqref{E:ss_initV}, and a corresponding estimate of the region of attraction, 
 an iterative process called \textit{expanding interior algorithm}, \cite{Wloszek:2003,Anghel:2013}, is used to iteratively enlarge the estimate of the region of attraction by finding a better Lyapunov function at each step of the algorithm.} At the completion of this iterative step, the stability of each isolated subsystem (assuming no interaction) is quantified by its Lyapunov function $V_i(z_i)$, with an estimation of the boundary of the domain of attraction given by $\mathcal{R}_{A,i} = \left\lbrace z_i\in\mathbb{R}^{m_i}\left| G_i(z_i) =0, V_i(z_i)\leq 1\right.\right\rbrace$.

The Lyapunov level-sets can be used to express the strength of a disturbance. The equilibrium point of the system at origin corresponds to the level set $V_i(0)=0,\forall i$.  If there is a disturbance from this equilibrium point, the states of the system would move to some point $x(0)$ away from the origin. This disturbed initial condition would result in positive level-sets $V_i(z_i(0))=\gamma_i^0\in\left(0,1\right]$ for some or all of the subsystems. 
A necessary and sufficient condition of asymptotic stability can then be translated into the condition
\begin{align}\label{E:cond_asymp}
\forall i, ~V_i(z_i(0))=\gamma_i^0\implies\forall i, ~\lim_{t\rightarrow +\infty}{V}_i(z_i(t))=0
\end{align}

In the rest of the article, we present SOS algorithms to test stability conditions and design local (subsystem-level) control laws to achieve asymptotic stability.

\section{\MakeUppercase{Stability under Interactions}}\label{S:stability}
It is assumed that the isolated subsystems in \eqref{E:ss_dae} are all (locally) asymptotically stable, and that there exist subsystem Lyapunov functions $V_i(z_i),\forall i$. {The estimated} region of attraction of the interconnected system under no interaction, $\mathcal{R}_A^0$, is given by the cross-product of the regions of attraction of the isolated subsystems, $\mathcal{R}_{A,i}$, which are defined as sub-unity-level sets of the corresponding (properly scaled) subsystem Lyapunov functions (as in (\ref{E:ROA})-(\ref{E:ROA2})), i.e.
\begin{subequations}\label{E:ROA_isol}
\begin{align}
\mathcal{R}_A^{0} :=&~\mathcal{R}_{A,1}\times \mathcal{R}_{A,2}\times \dots \times\mathcal{R}_{A,m}  \notag \\
\text{where},~ \mathcal{R}_{A,i} = &\left\lbrace z_i\in\mathbb{R}^{m_i}\left| G_i(z_i) = 0, V_i(z_i)\leq 1\right.\right\rbrace, 
~\forall i \, . \notag
\end{align}
\end{subequations}
In presence of non-zero interactions, the resulting ROA would be different. If there exists a Lyapunov function for the interconnected system, the ROA for the whole system could be expressed as some sub-level set of that Lyapunov function. While it is very hard to obtain a scalar Lyapunov function for the full interconnected system, one could use vector Lyapunov function approach to obtain certification of stability in a scalable way.

In this present work, we choose not to impose any further restriction on the Lyapunov functions $V_i(z_i)$ than requiring that those are in polynomial forms, and concern ourselves with asymptotic stability. 
We now present a distributed iterative procedure which can be used to certify asymptotic stability in a domain defined by sub-level sets of the subsystem Lyapunov functions.

\subsection{Algorithmic Test of Asymptotic Stabiltiy}\label{S:asymptotic}
Before proceeding to explaining our algorithm, let us first note the following result:
\begin{lemma}\label{L:asymptotic}
Suppose, for all $i\in\left\lbrace 1,2,\dots, S \right\rbrace$, there exists a strictly monotonically decreasing sequence of scalars $\left\lbrace \epsilon_i^k\right\rbrace,k\in\left\lbrace 0,1,2,\dots\right\rbrace$, such that
\begin{subequations}\label{E:cond_asymptotic}
\begin{align}
\forall i,k,~\dot{V}_i(z_i)&:=\nabla{V}_i(z_i)^T\left(F_i(z_i)+H_i(z)\right)\!<\!0,~\forall z \in\mathcal{D}_i^k\\
\text{where,}~\mathcal{D}_i^k &:=\!\left\lbrace\!\! z \in\mathbb{R}^{\bar{m}_i} \left \vert \!\!\!
\begin{array}{c}
\epsilon_i^{k+1}\leq V_i(z_i)\leq\epsilon_i^k,\\
V_j(z_j)\leq\epsilon_j^k~    \forall j\in\mathcal{N}_i\setminus\{i\}\\
G_j(z_j) = 0   ~\forall j\in\mathcal{N}_i
\end{array} 
\right.\!\!\!\!\!\!\right\rbrace \label{E:Dk}
\end{align}
\end{subequations}
Then the system \eqref{E:dae} is asymptotically stable in the domain $\left\lbrace z \in\mathbb{R}^m \left| \bigcap_{i=1}^S V_i(z_i)\leq \epsilon_{i}^0\right.\right\rbrace$, if $\lim_{k\rightarrow +\infty} \epsilon_i^k= 0,\forall i$
If the limit condition does not hold, we can only guarantee stability in the sense of Lyapunov \cite{Slotine:1991,Lyapunov:1892}.
\end{lemma}
\begin{proof}
Please refer to Appendix~\ref{A:proof}.
\end{proof}

Using the Lemma~\ref{L:asymptotic} we can devise a simple iterative SOS algorithm to certify whether or not a domain $\mathcal{D}$ defined by
\begin{align}\label{E:D}
\mathcal{D}:=\left\lbrace z \in\mathbb{R}^m 
                         \left| \bigcap_{i=1}^m 
                                \left\lbrace 
                                V_i(z_i)\leq \gamma_i^0 , G_i(z_i) = 0 
                                \right\rbrace 
                           \right. 
                       \right\rbrace, 
\end{align}
for some scalars $\gamma_i^0\in\left(0,1\right],\forall i$, is a region of asymptotic stability for the system in \eqref{E:dae}. It is to be noted that, using the Putinar's Positivstellensatz theorem (Theorem~\ref{T:Putinar}), the condition in (\ref{E:cond_asymptotic}) essentially translates into equivalent SOS feasibility conditions
\begin{align}\label{E:cond_asymptotic_SOS}
&\forall i,k,~\exists \sigma_{i}^k,\sigma_{ij}^k\in\Sigma_{\bar{m}_i}, 
     \lambda_j^k \in \mathcal{R}_{\bar{m}_i}^{p_j} ~\text{s.t.}  \notag \\
&-\nabla{V}_i^T\left(F_i+H_i\right) -\sigma_{i}^k\left(V_i-\epsilon_i^{k+1}\right)  \notag  \\
&  \qquad - \sum_{j\in\mathcal{N}_i}\sigma_{ij}^k\left(\epsilon_j^k-V_j\right)  
                 -   \sum_{j\in\mathcal{N}_i} (\lambda_j^k)^T G_j  \in\Sigma_{\bar{m}_i} \\
&\text{where,}~\bar{m}_i = \sum_{j\in\mathcal{N}_i}m_j.\notag
\end{align}
The algorithmic steps to ascertain asymptotic stability are as outlined below:
\begin{enumerate}
\item We initialize $\epsilon_i^0=\gamma_i^0,\forall i\in\left\lbrace 1,2,\dots,S \right \rbrace$, and choose a sufficiently small $\bar{\epsilon}\in\mathbb{R}^+$.

\item\label{I:iteration} At the start of the $k$-th iteration loop, we assume to {know} the scalars $\left\lbrace \epsilon_i^0,\epsilon_i^1,\dots,\epsilon_i^k\right\rbrace,\forall i$, and our aim is to compute the scalars {$\epsilon_i^{k+1},~\forall i$} such that (\ref{E:cond_asymptotic_SOS}) holds. Essentially we want to solve the optimization problem,
\begin{equation}\label{E:max_epsilon}
\forall i,~ \min_{\sigma_{i}^k,\sigma_{ij}^k,\lambda_j^k} ~\epsilon_i^{k+1}
\quad \text{s.t. \eqref{E:cond_asymptotic_SOS} holds.}
\end{equation}
This is solved by performing a bisection search for minimum $\epsilon_i^{k+1}$ over the range $\left[0,\epsilon_i^k\right]$. 

If (\ref{E:max_epsilon}) is infeasible at $0$-th iteration
 for any $i\in\left\lbrace 1,2,\dots, S \right\rbrace$, we conclude that the system cannot be guaranteed to be asymptotically stable in $\mathcal{D}$, and abort the iteration. Otherwise we move on to step \ref{I:conclusion}.

\item\label{I:conclusion} If $\left(\epsilon_i^k-\epsilon_i^{k+1}\right)\geq\bar{\epsilon},\forall i$, we continue from step \ref{I:iteration} for the ($k$+1)-th iteration loop. Otherwise we stop the iteration deciding that the limits of the sequences $\left\lbrace\epsilon_i^k\right\rbrace,\forall i,$ have been attained. Further, if the {limits are all zero, we certify asymptotic stability in $\mathcal{D}$}.
\end{enumerate}

\subsection{Remarks}\label{S:asymptotic_remarks}
The algorithm presented in Sec.~\ref{S:asymptotic} describes how one can determine asymptotic stability of an interconnected system in a domain $\mathcal{D}$ defined by the subsystem sub-level sets. This test can be performed locally, and in a parallel way, at each subsystem level. The Lyapunov functions $V_i$'s are to be found before the start of the analysis, and communicated to the neighboring subsystems. Then during each analysis, it is assumed that the neighboring subsystems can communicate with each other the computed sequences $\left\lbrace \epsilon_i^k\right\rbrace$ {in} real-time. With the help of {the} stored Lyapunov functions, and the updated $\left\lbrace \epsilon_i^k\right\rbrace$ of the neighbors, each subsystem will continue the iterative process outlined in Sec.~\ref{S:asymptotic}. Since only the neighbor information is required, this algorithm is reasonably scalable with respect to the size of the full interconnected system. Moreover, the algorithm motivates the design of a distributed control strategy that can ascertain asymptotic stability.

\section{\MakeUppercase{Local Control}}\label{S:control}
In this section, we discuss design of a \textit{local} and \textit{minimal} control strategy such that the system in \eqref{E:dae} is asymptotically stable in a domain $\mathcal{D}$ defined in (\ref{E:D}). We use the term \textit{minimal} to suggest that the control be applied only in certain regions, and not everywhere, in the state space, while by the term \textit{local} we suggest that the control is computable and implementable on a subsystem level.  

We envision the control to be computed by each subsystem at each iteration loop. At $k$-th iteration, $\forall k\in\left\lbrace 0,1,2,\dots\right\rbrace$, the  $i$-th subsystem, $\forall i\in\left\lbrace 1,2,\dots, S \right\rbrace$ performs the following tasks:
\begin{enumerate}
\item It identifies if it belongs to the following set 
\begin{align}
\mathcal{U}^k&:= \left\lbrace i\in\left\lbrace 1,\dots,S \right\rbrace\left\vert 
\begin{array}{c}\nabla{V}_i^T\left(F_i+G_i\right)\geq0, \\
     V_i=\epsilon_i^k,\\
     V_j\leq\epsilon_j^k ~\forall j\in\mathcal{N}_i\setminus\{i\} \\
    G_j(z_j) = 0  ~\forall j\in\mathcal{N}_i
\end{array}\right.\right\rbrace \notag
\end{align}
which can be checked locally. If $i\notin\mathcal{U}^k$, control is not necessary and it sets $U_i^k\equiv0$ and proceeds to task~\ref{I:control_dyn}. If, however, $i\in\mathcal{U}^k$, it proceeds to task~\ref{I:Fi} to compute a control law.

\item\label{I:Fi} If $i\in\mathcal{U}^k$, a polynomial state-feedback control law $U_i^k:\mathbb{R}^{m_i}\rightarrow\mathbb{R}^{m_i}$, $U_i^k(0 )= 0$, is computed such that
\begin{align}
 \nabla{V}_i^T (F_i+ & G_i + U_i^k )< 0\, , \forall z  \in \mathcal{I}_i^k  \, ,\label{E:Fi_K} \\
 \text{where} ~\mathcal{I}_i^k & := 
 \left\lbrace 
 z \in \mathbb{R}^{\bar{m}_i} 
 \left\vert
\begin{array}{c}
V_i=\epsilon_i^k, \\
V_j\leq\epsilon_j^k~\forall j\in\mathcal{N}_i \setminus\{i\} \\
 G_j(z_j) = 0 ~\forall j \in\mathcal{N}_i
\end{array}
\right. 
\right\rbrace \notag
\end{align}
This produces the equivalent SOS condition,
\begin{align}
&- \nabla{V}_i^T \left( F_i + G_i + U_i^k \right)  - \rho_{i}^k \left( \epsilon_i^{k} - V_i \right)  \label{E:sos_K} \\
&  \qquad - \sum_{j\in\mathcal{N}_i\setminus\{i\}} \sigma_{ij}^k \left( \epsilon_j^k - V_j \right)
  -  \sum_{j\in\mathcal{N}_i} (\lambda_j^k)^T G_j 
           \in \Sigma_{\bar{n}_i}\notag\\
&\rho_{i}^k\in\mathcal{R}_{\bar{m}_i},~\sigma_{ij}^k\in\Sigma_{\bar{m}_i}\forall j\neq i, \lambda_j^k \in \mathcal{R}_{\bar{m}_i}^{p_j}, 
        ~\bar{m}_i = \sum_{j\in\mathcal{N}_i} m_j \notag
\end{align}

\item\label{I:control_dyn} Finally, it performs the search over minimum $\epsilon_i^{k+1}$, as in \eqref{E:max_epsilon}, with the un-controlled subsystem dynamics $\left(F_i+G_i\right)$ in the feasibility condition \eqref{E:cond_asymptotic_SOS}  replaced by the controlled dynamics $\left(F_i+G_i+U_i^k\right)$.
\end{enumerate}

To summarize, each subsystem $i$ computes control laws $U_i^k:\mathbb{R}^{m_i}\rightarrow\mathbb{R}^{m_i}$, with $U_i^k(0 )= 0$, during each $k$-th iteration, so that the subsystem dynamics under control becomes:
\begin{align}
\forall i,~\forall k, ~&\forall x\in\mathcal{D}_i^k,\notag\\
\dot{z}_i &= \left\lbrace \begin{array}{ll}F_i(z_i) + H_i(z_i), ~G_i(z_i) = 0, &i\notin\mathcal{U}^k\\
			F_i(z_i) + H_i(z_i)+U_i^k(z_i), ~G_i(z_i) = 0,&i\in\mathcal{U}^k\end{array}\right. \notag
\end{align}
where $\mathcal{D}_i^k$ were defined in (\ref{E:Dk}).

{\subsection{Remarks}
Often it is important to impose certain additional constraints on the possible control laws, such as bounds on the control effort. Although control bounds can be easily incorporated in the SOS formulation, we decide to keep that for future studies. 
\rem{We note,} However, that since we apply controls $U_i^k$ only on certain subsystems $i\in\mathcal{U}^k$, and in certain domains $\mathcal{D}_i^k\subseteq\mathcal{D}$, the control effort would be reasonably bounded.}

\section{\MakeUppercase{RESULTS}}\label{S:results}
Let us describe the model of the interconnected system that we use here, and two examples to illustrate the applications of the stability analysis algorithm and control design.
 
 \subsection{Model Description}\label{S:model}
 
 {W}{e} will consider the network preserving model with linear frequency dependent real power loads introduced in \cite{Bergen:1981,Hill:1982}.
 The network consists of  $G$ generators and $L$ load buses  connected by transmission lines. 
 We number the load buses $1,2,\ldots, L$ 
 and the generator buses $L+1,\ldots, n$, where $n=L+G$ is the total number of nodes.
 The node voltages are denoted by $
E_1 \angle \delta_1, \ldots, E_n \angle \delta_n$, where $\delta_1,
\ldots, \delta_n$ are the  phase angles 
and the magnitudes $E_1, \ldots, E_n$ are
assumed constant.
The electrical power injected into network at node $i$ is
\begin{equation}\label{E:Pe}
\begin{split}
 P_{Ei} (\delta) &= E_i^2 G_{ii} +  \sum_{j, j\neq i} E_i E_j   Y_{ij} \cos (\delta_i - \delta_j - \theta_{ij})  
\end{split}
\end{equation}
where  $Y_{ij}$ is the modulus, and $\theta_{ij}$ the phase angle, of the transfer admittance 
between nodes $i$ and $j$. 
%
 
 For  small frequency variations around the operating point  $P_{D_i}$ 
 the dynamics of the load nodes  are described by
\begin{equation}\label{E:load_dyn}
  D _i   \dot{\delta}_{i}  =   -  P_{D_i}  - P_{E_i}(\delta) \, , \, i=1,\ldots,L\, ,
\end{equation}
  where  $D_i > 0$ is the load-frequency coefficient and $P_{D_i} > 0  $ 
  is the real power drawn at the load buses.
Each generator  dynamics are
modeled by the swing equations. Thus, for  $i = 1,\ldots, G,$
\begin{subequations}\label{E:gen_dyn}
\begin{align}
  \dot{\delta}_{L+i} &= \omega_{i} \, ,  \\
  M_i \dot{\omega}_{i}  +D _{L+i} \omega_{i}  &=  P_{M_i}   - P_{E_{L+i}}(\delta) \, ,
\end{align}
\end{subequations}
where   $M_i  > 0$ is the generator inertia constant,  
$ D_{L+i} > 0   $ is the generator damping coefficient,
 $ P_{M_i} >0 $ is the mechanical power input.

We  assume that the dynamical system has  a stable equilibrium point  given by  
$( \delta_{s},\omega_{s} = 0)$ where $\delta_{s} $ is the solution of the following set of nonlinear equations,
\begin{align}
-  P_{D_i}  - P_{E_i}(\delta_s) &= 0\, , \text{for}\, \, i=1,\ldots,L \\
P_{M_i} - P_{E_{L+i}}(\delta_s) &= 0 \, , \text{for}\, \, i= 1,\ldots,G
\end{align}

The state space of the dynamical system is described by the relative angles 
$\delta_{in} = \delta_i - \delta_n$, for $i=1,\ldots,L+G-1$, with respect to a
reference node $n$ (generator $G$, considered to have the largest inertia), 
 and {relative generator speeds 
$\omega_{in}=\omega_i - \omega_n$, for $i=1,\ldots,G-1$, for all generators except the reference generator 
for which we consider the absolute speed $\omega_n=\omega_G$}. (We have considered that the ratio $D_{L+i}/M_i$ 
is uniform.) 
%
 The dynamics of the relative angles  and speeds are obvious  and are not explicitly presented here.
Finally,  we make the following change of variables,
$\delta \rightarrow \delta + \delta_{s}$, in  \eqref{E:load_dyn} and \eqref{E:gen_dyn}
 in order to transfer the stable equilibrium point to the origin in phase space.

\subsection{Recasting the Power System Dynamics}
\label{sec:recasting_method}

SOS programming  methods cannot be directly applied  to   study the stability of power system  models because their dynamics contain trigonometric nonlinearities and are not polynomial.  For this reason a systematic methodology to recast their dynamics into a polynomial form is necessary \cite{Antonis:2002, Antonis:2005}.  The recasting 
introduces  a number of equality constraints restricting the states to a manifold having the original state dimension. 
For the network preserving power system model  recasting is achieved  by  a non-linear change of variables,
{\begin{subequations}\label{E:variable_recasting}
\begin{align}
\{ \delta_{in} \} & \rightarrow 
    \left\lbrace\!\!
    \begin{array}{c}
     z_{2i-1} = \sin(\delta_{in}) \\
     z_{2i} = 1 - \cos(\delta_{in})
     \end{array}
    \!\!\! \right\rbrace\, , \,  i=\!\!1,\ldots, L\! +\!G\!-\!1\!\, , \\
 \{  \omega_{in} \}  
&\rightarrow  \{ z_{2(L+G-1)+i} = \omega_{in} \} \, , \, i=1,\ldots, G-1\, , \\
 \{  \omega_{i} \}  
&\rightarrow  \{ z_{2(L+G-1)+i} = \omega_{i} \} \, , \, i=G\, ,
\end{align}
\end{subequations}}
and the introduction of the following constraints
\begin{equation}\label{E:variable_constraints}
 z_{2i-1}^2 + z_{2i}^2 - 2 z_{2i} = 0\, , \, i=1,\ldots, L  + G - 1\, .
\end{equation}
Thus, recasting produces a dynamical system with a larger state dimension, $z \in  \mathbb{R}^m$,
where $m= 2(L+G-1)+G$ and  introduces  $ q= L+G-1$  equality constraints.
 The stable equilibrium point of the original system
 is mapped to $z_s=0$.
 The original system dynamics 
are recasted into the polynomial differential algebraic equations \eqref{E:dae}.

\subsection{Test Case: IEEE 9-bus System}
\label{sec:test_case}

 \begin{figure}[thpb]
      \centering
	\includegraphics[width=2.5in]{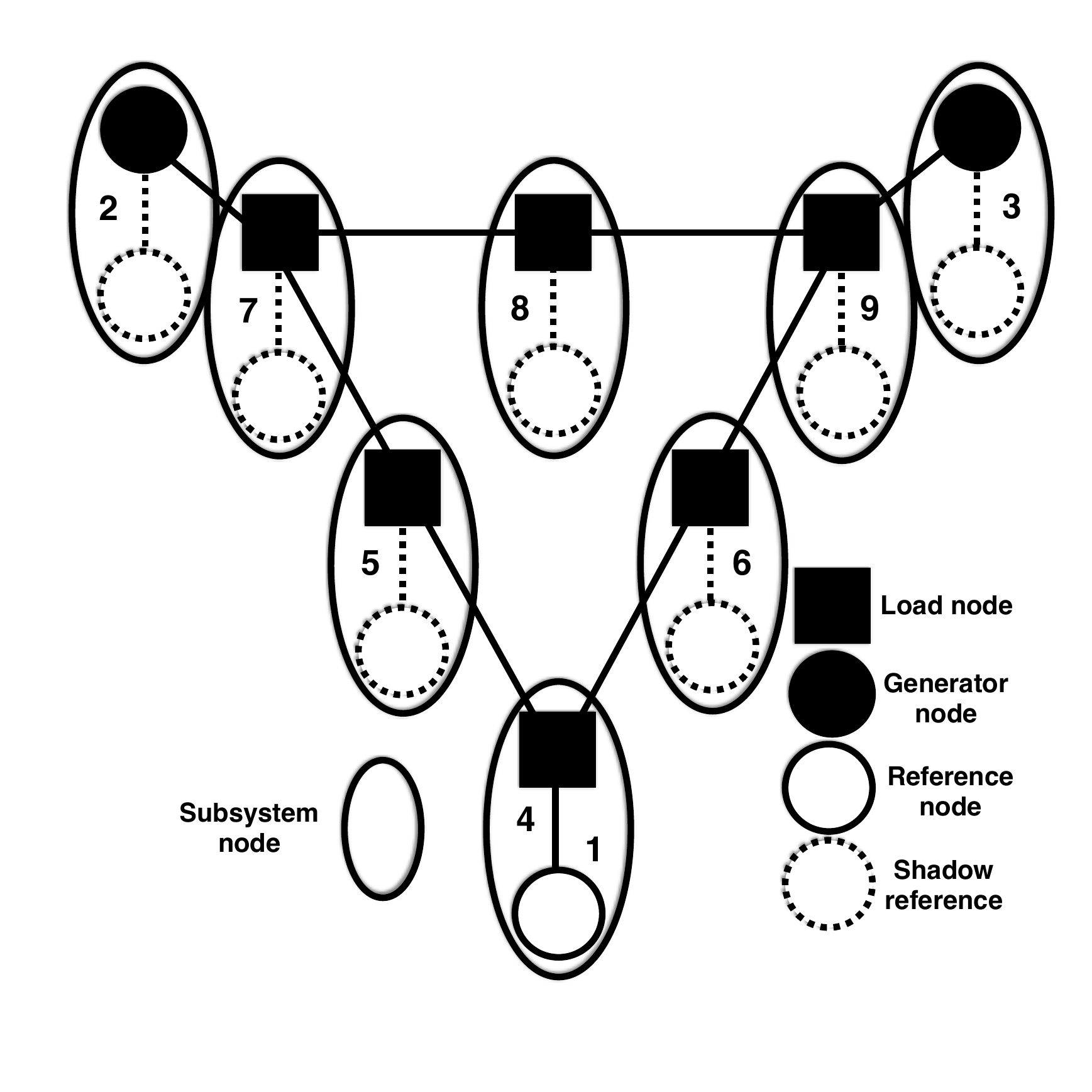} 
	      \vspace*{-0.8cm}\caption{Network of three generators and six load nodes. We perform
	      an overlapping decomposition  in which the speed dynamics of the 
	      reference node (generator node 3) is shared with all the subsystems. This 
	      results in the following overlapping subsystems: {$S_1=\{4,1\}, S_2=\{5,1\},
	      S_3=\{6,1\},S_4=\{7,1\},S_5=\{8,1\},S_6=\{9,1\},S_7=\{2,1\},S_8=\{3,1\}$}.
	      }
      \label{F:net5}
   \end{figure}

\begin{figure*}[thpb]
\centering
\subfigure[ROA of subsystem 2 (node 5)]{
\includegraphics[width=2.5in]{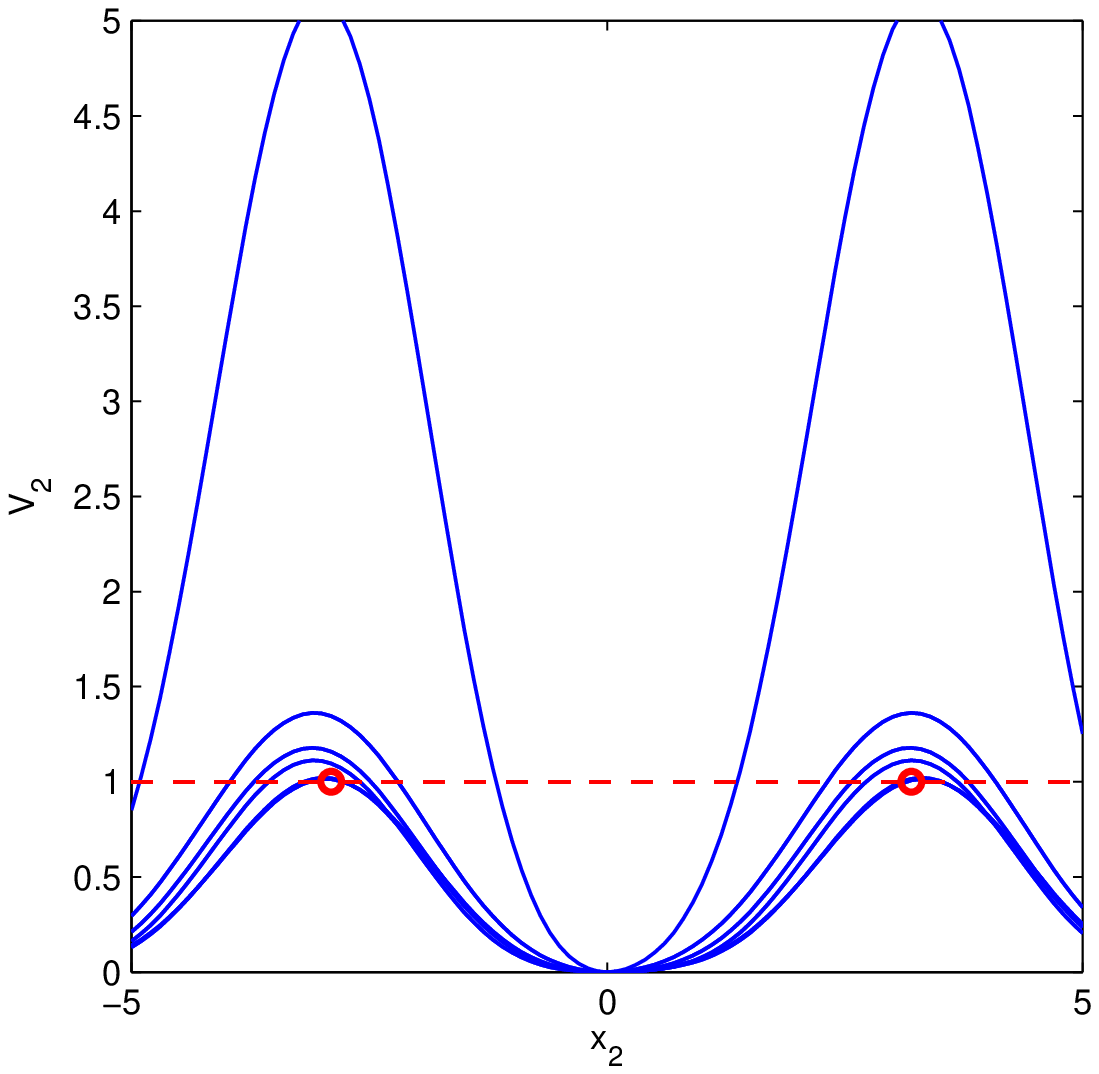}\label{F:roa_ss2}
}\quad
\subfigure[ROA of subsystem 7 (node 2)]{
\includegraphics[width=2.5in]{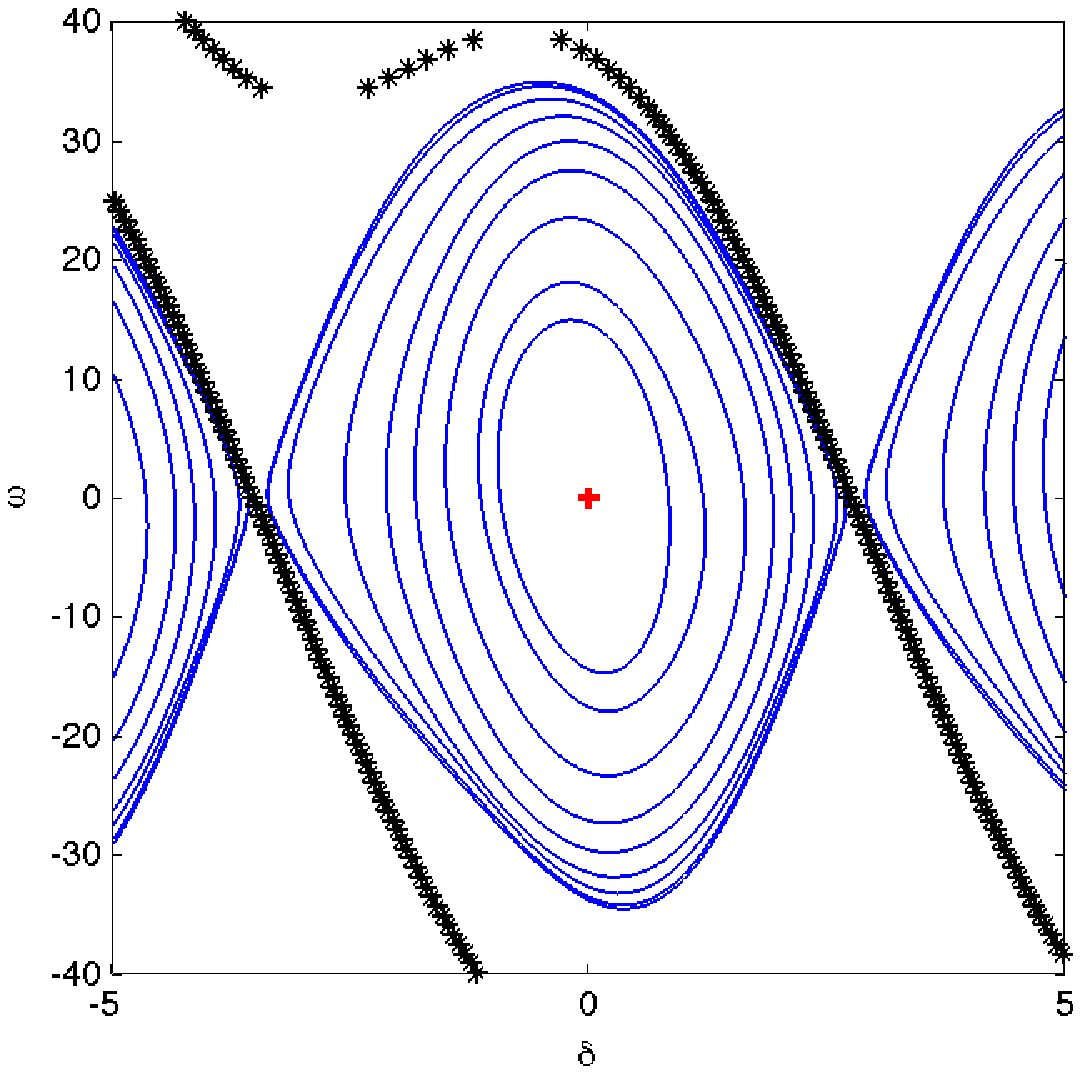}\label{F:roa_ss7}
}\caption[]{ Estimates of regions of attraction of isolated subsystems using expanding interior algorithm.}
\label{F:roa}
\end{figure*}
\begin{figure*}[thpb]
\centering
\subfigure[Load angles, before control]{
\includegraphics[width=2.2in]{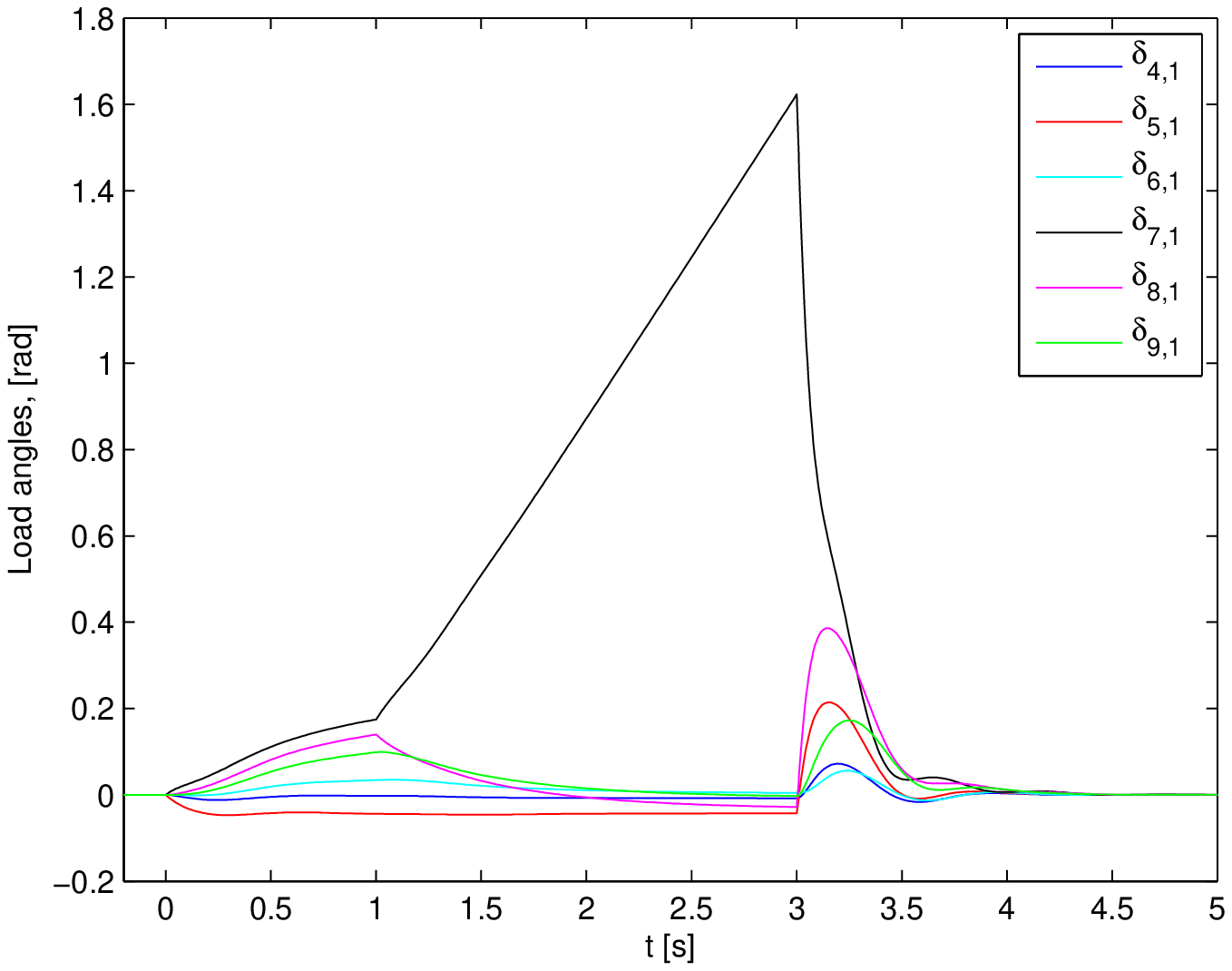}\label{F:prestates}
}\hspace{0.01in}
\subfigure[Generator states, before control]{
\includegraphics[width=2.2in]{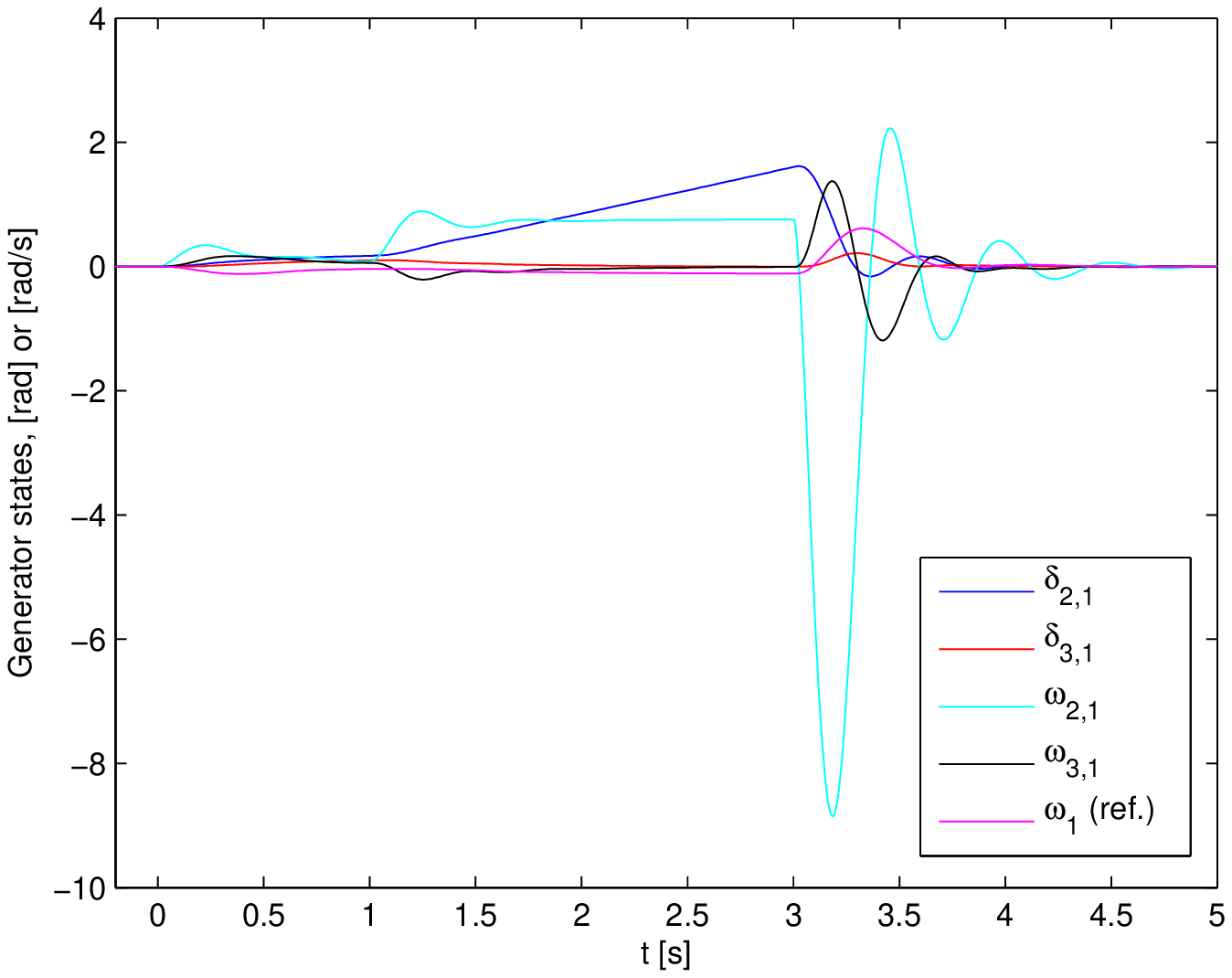}\label{F:prestates2}
}\hspace{0.01in}
\subfigure[Subsystem Lyapunov functions]{
\includegraphics[width=2.2in]{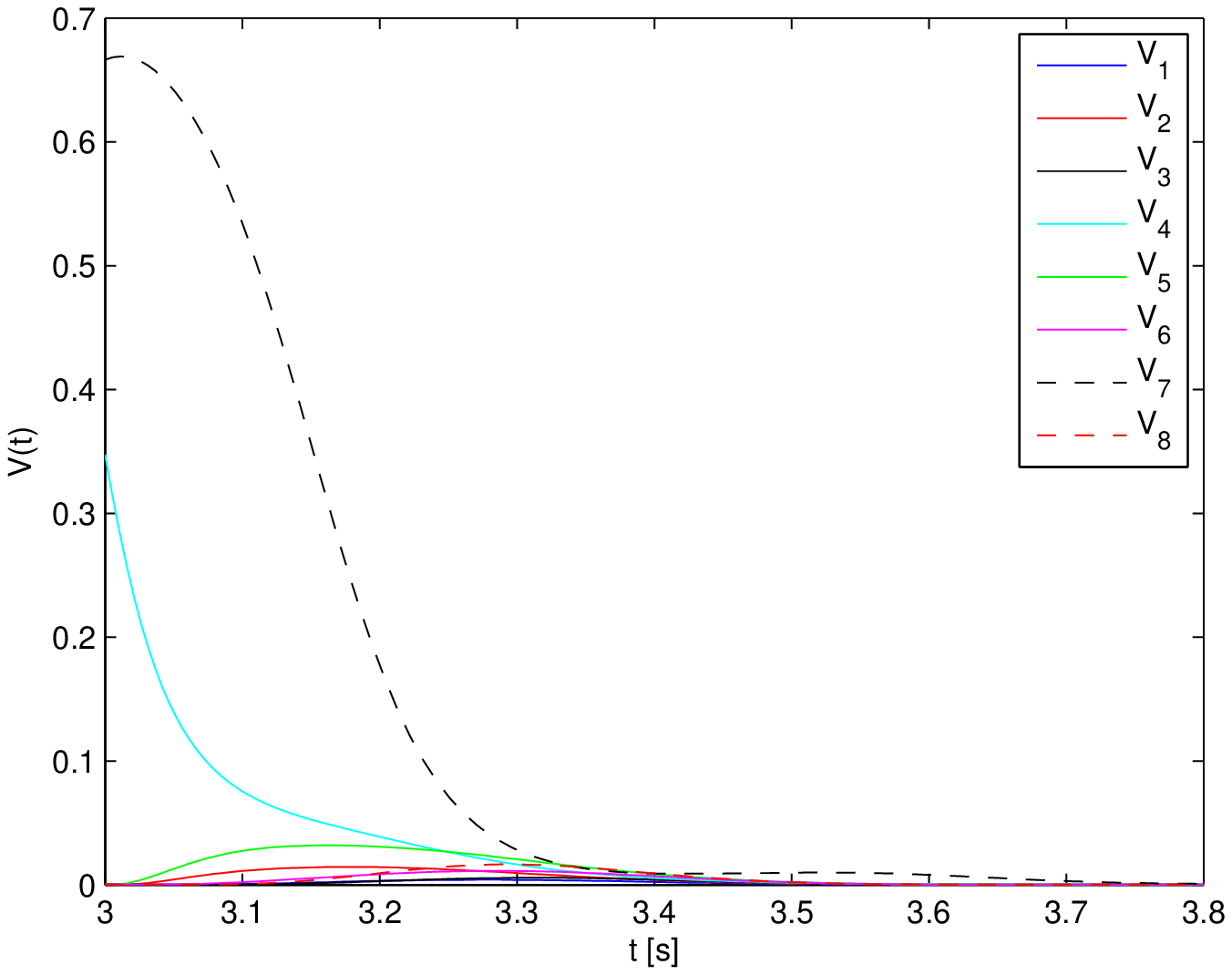}\label{F:preV}
}\caption[Optional caption for list of figures]{System states and Lyapunov functions under a disturbance, before control.}
\label{F:pre_control}
\end{figure*}
\begin{table*}[thpb]
\caption{Algorithm certifies asymptotic stability, under control}
\label{Tab:bounded}
\begin{center}
\begin{tabular}{|c|c|c|c|c|c|c|c|c|}
\hline
$k$ & $\epsilon_1^k$ & $\epsilon_2^k$ & $\epsilon_3^k$ & $\epsilon_4^k$ & $\epsilon_5^k$ & $\epsilon_6^k$ & $\epsilon_7^k$ & $\epsilon_8^k$\\
\hline
0 & 0.0001$^*$  &  0.0007$^*$ &   0.0002$^*$ &   0.3471 &   0.0003$^*$  &  0.0001$^*$  &  0.6663$^*$  &  0.0001$^*$\\
\hline
1 & 0.0000 &   0.0000  &  0.0002 &   0.0312 &   0.0000 &   0.0001 &   0.4451  &  0.0000\\
\hline
2 & 0.0000 &   0.0000 &   0.0001  &  0.0203 &   0.0000 &   0.0000 &   0.4260 &   0.0001\\
\hline
3 & 0.0000 &   0.0000  &  0.0001 &   0.0195  &  0.0000 &   0.0000 &   0.3478 &   0.0000\\
\hline
4 & 0.0000  &  0.0000 &   0.0000  &  0.0154  &  0.0000  &  0.0000 &   0.3383 &   0.0000\\
\hline
\vdots & \vdots & \vdots & \vdots & \vdots & \vdots & \vdots & \vdots & \vdots \\
\hline
11 & 0.0000 &   0.0000 &   0.0000 &   0.0103 &   0.0000 &   0.0000 &   0.2481 &   0.0000\\
\hline
12 & 0.0000 &   0.0000 &   0.0000 &   0.0010 &   0.0000 &   0.0000 &   0.2374 &   0.0000\\
\hline
13 & 0.0000 &   0.0000 &   0.0000 &   0.0001 &   0.0000 &   0.0000 &   0.0492 &   0.0000\\
\hline
14 & 0.0000 &   0.0000 &   0.0000 &   0.0000 &   0.0000 &   0.0000 &   0.0094 &   0.0000\\
\hline
15 & 0.0000 &   0.0000 &   0.0000 &   0.0000 &   0.0000 &   0.0000 &   0.0000 &   0.0000\\
\hline
\end{tabular}
\end{center}
\end{table*}
\begin{figure*}[thpb]
\centering
\subfigure[Load angles, with control at Gen.~2]{
\includegraphics[width=2.2in]{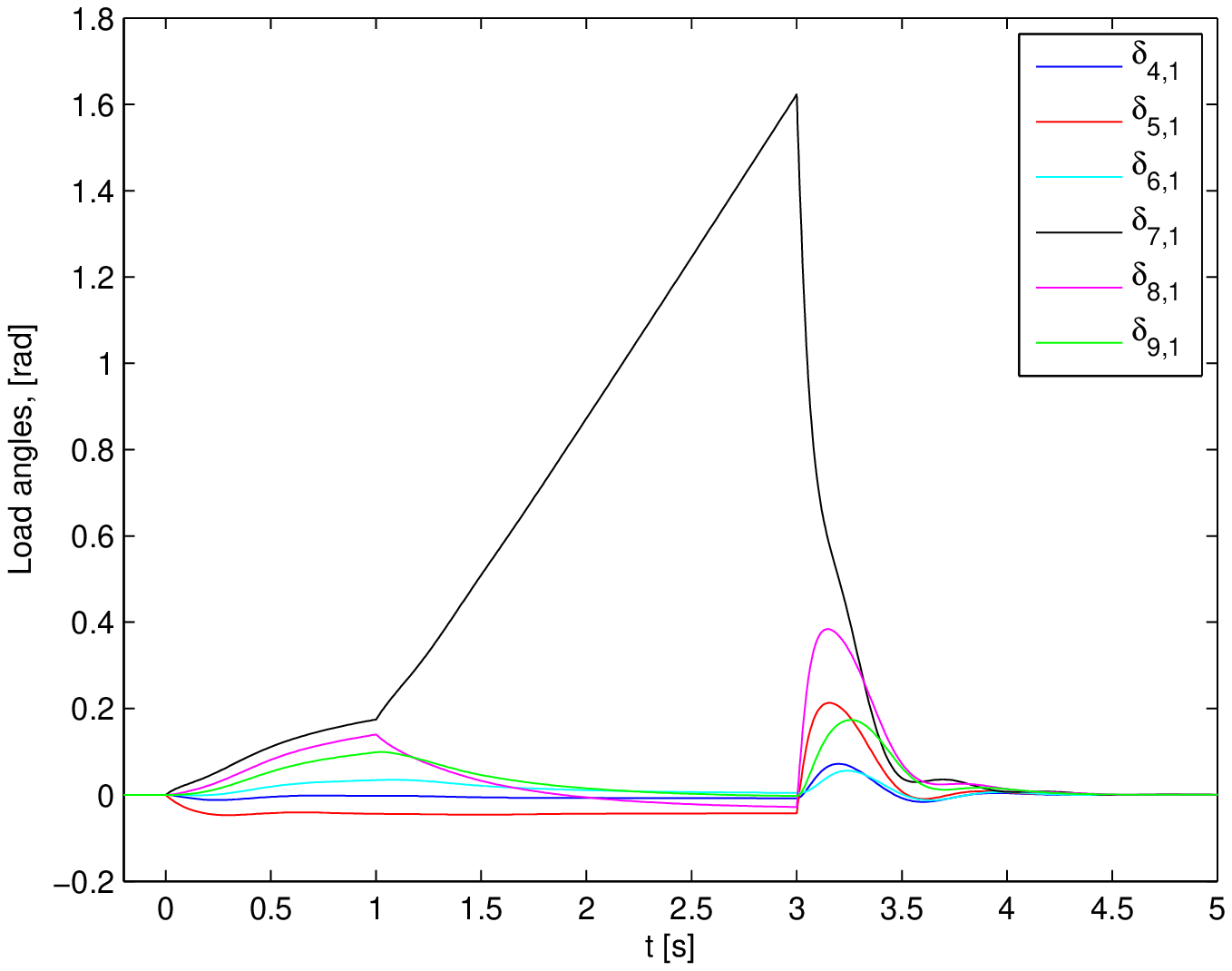}\label{F:poststates}
}\hspace{0.01in}
\subfigure[Generator states, with control at Gen.~2]{
\includegraphics[width=2.2in]{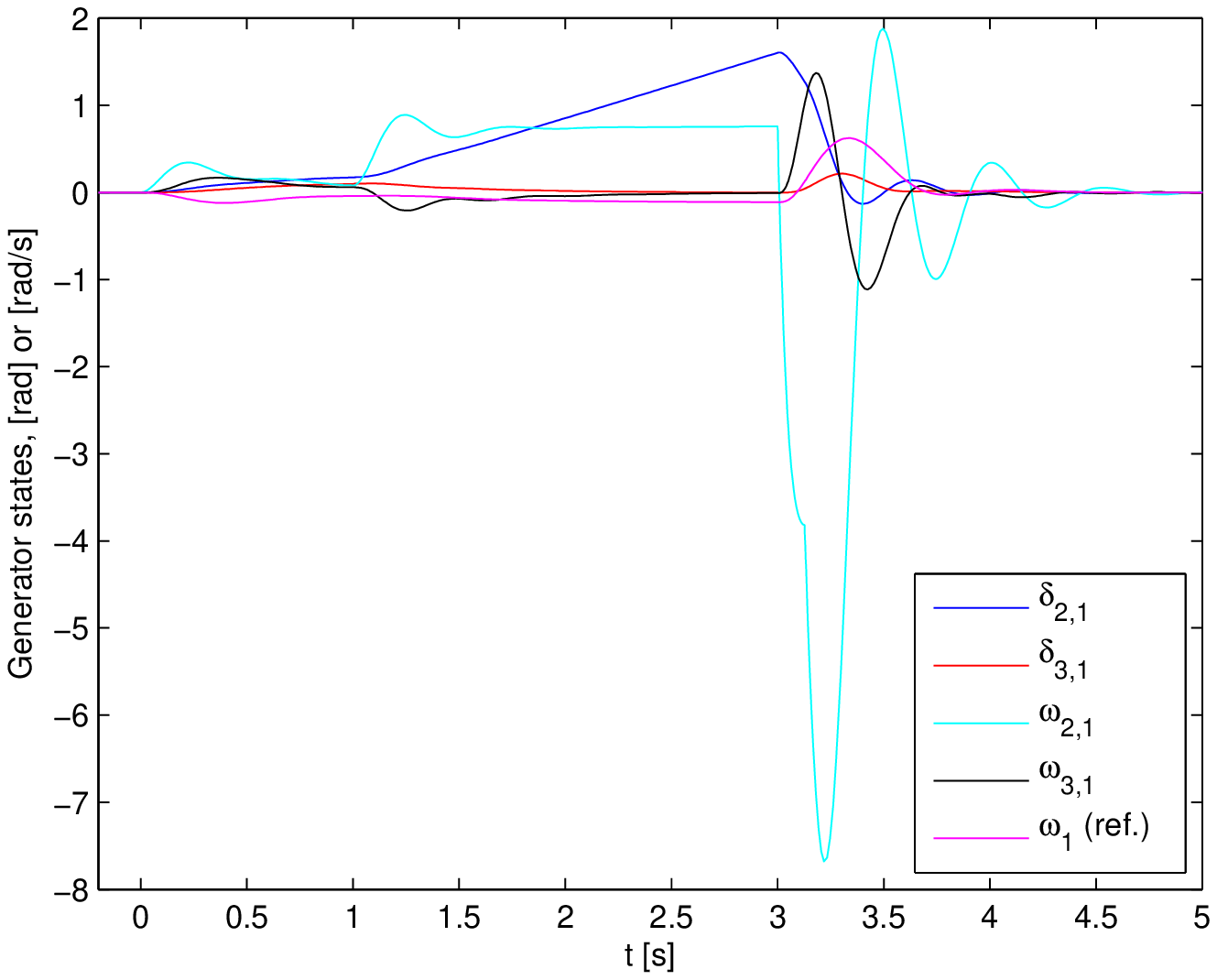}\label{F:poststates2}
}\hspace{0.01in}
\subfigure[Subsystem Lyapunov functions]{
\includegraphics[width=2.2in]{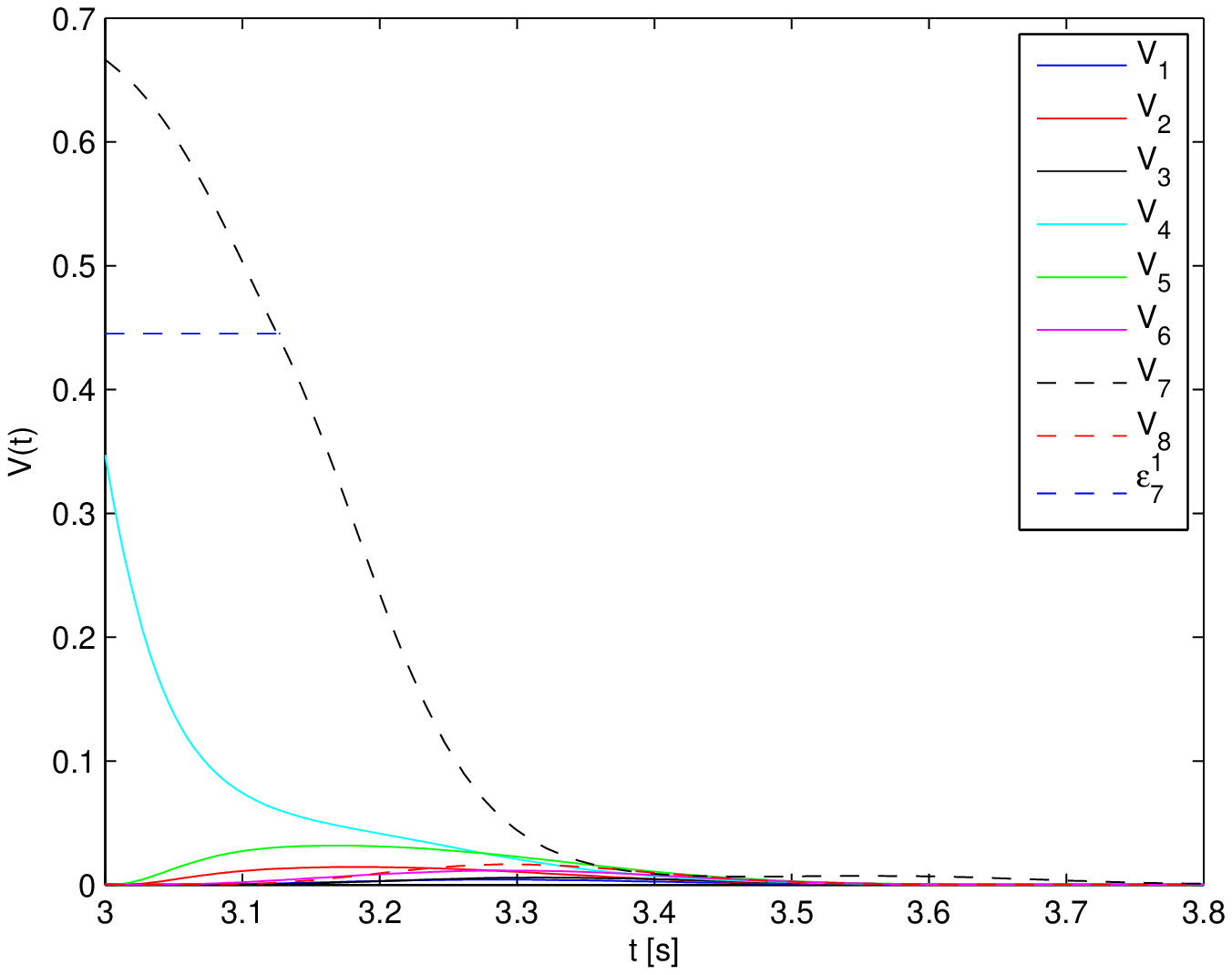}\label{F:postV}
}\caption[Optional caption for list of figures]{System states and Lyapunov functions under local and temporary control applied at generator 2, for the duration $t\in[3s,\,3.127s]$.}
\label{F:post_control}
\end{figure*}
We will be using the Western System Coordinating Council (WSCC) 9-bus system (commonly, the IEEE 9-bus model \cite{Amer:2000}), for our analysis. To better illustrate the scope of our approach, we redistribute the loads so that each of the nodes $4,\,6$ and $9$ in the network has certain loads, as shown in Fig.~\ref{F:net5}. Thus, in our modified network, each node has some associated dynamics. Further, using an overlapping decomposition \cite{Ikeda:1980}, we construct eight subsystems where each subsystem is composed of either a load node or a (non-reference) generator node, along with the reference generator node. Essentially this requires that the reference generator speed value is to be communicated to all the subsystems, in real-time\footnote{Communication bandwidth and time-delays associated with such a model are important issues that are beyond the scope of the present work.}. We choose $D_{L+i}=5M_i$ for the generators and select $D_i$ randomly from $[1,2]$ for the loads.

We use the expanding interior algorithm to find estimates of the ROA for each isolated subsystem. Fig.~\ref{F:roa} shows the evolution of ROA estimates projected on $\omega_n=0$. In Fig.~\ref{F:roa_ss2} we see how the level set of the Lyapunov function for a load node evolves as the ROA estimate, computed as in (\ref{E:ROA})-(\ref{E:ROA2}), grows. Fig.~\ref{F:roa_ss7} shows a comparison of the true ROA for a generator node with a sequence of its estimates where the outermost `blue' contour represents the final estimate\footnote{The ROA estimates are obtained using quadratic polynomials.}.

Any randomly picked initial condition, $x(0)\in\mathcal{R}_A^0$, where $\mathcal{R}_A^0$ is defined in (\ref{E:ROA_isol}), can be mapped into corresponding subsystem Lyapunov function level sets, $\gamma_i^0=V_i(z_i(0)),\forall i$. Then, by choosing $\epsilon_i^0=\gamma_i^0$, we apply the iterative stability analysis algorithm to determine whether or not the domain $\mathcal{D}$ in (\ref{E:D}) is a region of asymptotic stability, and if not, compute the necessary control by (\ref{E:Fi_K}). 

\subsection{Disturbance Analysis}

We assume that initially the network was operating at equilibrium. A disturbance was created by tripping the line between nodes 5-7 for the duration $t\in[0,\,3s]$ and also tripping the line between nodes 7-8 for $t\in[1s,\,3s]$, essentially disconnecting the nodes 7 and 2 from the rest of the network for $t\in[1s,\,3s]$. 
The end of the disturbance provides the initial condition for the stability study. The evolution of states and the Lyapunov functions from this initial condition (i.e., at $t=3s$) is shown in Fig.~\ref{F:pre_control}. The initial condition is asymptotically stable. Further, the local effect of the disturbance is clear from Fig.~\ref{F:preV}, where apart from $V_4(t)$ and $V_7(t)$, all other Lyapunov functions are  very small.

In Table~\ref{Tab:bounded} the results from an application of the stability analysis algorithm to the initial conditions are shown. The row corresponding to $k=0$ lists the initial level sets, while subsequent rows list the sequence $\{\epsilon_i^k\}$ for all the subsystems. The $^*$ shows when the algorithm feels the necessity of applying control\footnote{We chose to seek only linear controllers in the $z$ variables which are applied on the angle dynamics for the loads, and on the speed dynamics for the generators. The controllers are nonlinear in the original state space. } to guarantee the decrease of level set. For example, at the first iteration step, the algorithm prescribes applying state-feedback control
\begin{align}\label{E:control_ss7}
\mathcal{U}_7^0 &= 5.1\cos\delta_{2,1}\!\!-\!38.4\omega_{2,1}\!-\!1.8\omega_1\!-\!71.4\sin\delta_{2,1}\!\!-\!5.1
\end{align}
which is applied to the speed dynamics  of the generator 2 of $S_7$. It can be seen from Fig.~\ref{F:preV} that initially $V_7(t)$ shows a slight increase before starting to decrease, which is why the control is prescribed. For the same reason, control is also suggested for subsystems other than $S_4$. However, other than $S_7$, control action can be safely ignored because the level sets are very close to zero.

After $k=1$, no control is deemed necessary by the algorithm, and finally after iteration $k=15$, the certificate of asymptotic stability is obtained. Fig.~\ref{F:post_control} shows the evolution of states and Lyapunov functions under the action of control applied at the generator 2. The control is applied only for the time during which $0.4451\leq V_7(t)\leq 0.6663$, which amounts to the time interval $t\in[3s,\,3.127s]$, as shown in Fig.~\ref{F:postV}. It is to be noted that under the control action, $\dot{V_7}(t)$ becomes negative at $t=3s$.

}

\section{CONCLUSIONS}\label{S:conclusion}
{In this work, we present a distributed algorithmic approach to infer the stability of an interconnected power system under a given disturbance. When this stability cannot be certified we design local and minimal control laws that guarantee asymptotic stability. The approach presented here is parallel and scalable. While this method is applied here considering dynamic loads in a network preserving power system model, it can be extended to more general, possibly more complex, dynamical representation of the power system network. Future work needs to address the issues of bounded control effort, voltage dynamics of loads and inclusion of available control mechanisms, such as speed governors. }


\appendix

\subsection{Proof of Lemma~\ref{L:asymptotic}}\label{A:proof}

{We note that since $\lim_{k\rightarrow +\infty} \epsilon_i^k= 0,\forall i$, 
\begin{align}\label{E:proof_delta}
\forall \delta\in\left(0,\min_i\epsilon_i^0\right], ~\exists K, ~\text{s.t.}~\epsilon_i^k<\delta~\forall k>K,\forall i.
\end{align}
Let us assume, without any loss of generality, that 
\begin{align}
\exists t_0\geq 0, ~\text{s.t.}~z(t_0)\in\left\lbrace z \in\mathbb{R}^m\left| 
 \bigcap_{i=1}^S  \left\lbrace V_i(z_i)\leq \epsilon_{i}^0 , G_i(z_i) = 0 \right \rbrace   \right.\right\rbrace \, . \notag
\end{align}
Then,
\begin{align}
\forall i, \quad& V_i(t) \leq \epsilon_i^0 + \int_{t_0}^t \dot{V}_i(\tau)d\tau,\quad \forall t\geq t_0 \notag \\
\implies & \exists~\! t_i^1 < t_0 + \left(\epsilon_i^1-\epsilon_i^0\right)/\bar{r}_i^1,\quad \bar{r}_i^1:=\sup_{x\in\mathcal{D}_i^1} \dot{V}_i(x)<0 \notag \\
	& \text{s.t.}~V_i(t)<\epsilon_i^1,~\forall t\geq t_i^1 \, . \notag
\end{align}
Hence, we can argue that,
\begin{align}
V_i(t)\leq\epsilon_i^0,&~\forall t\geq t_0,\forall i\notag\\
\implies &\exists~\! t^1:=\max_i t_i^1,~\text{s.t.}~V_i(t)\leq \epsilon_i^1,\forall t\geq t^1,\forall i \, . \notag
\end{align}
Following similar arguments it is easy to show that,
\begin{align}\label{E:proof_tk}
V_i(t)\leq\epsilon_i^0 &,~\forall t\geq t_0,~\forall i\notag\\
\implies &\forall k, ~\exists~\! t^k\geq t_0,~\text{s.t.}~V_i(t)\leq \epsilon_i^k,~\forall t\geq t^k,~\forall i \, .
\end{align}
Finally combining (\ref{E:proof_delta}) and (\ref{E:proof_tk}) we observe,
\begin{align}
\forall \delta\in\left(0,\min_i\epsilon_i^0\right],~\exists ~\! t^K\geq t_0, ~\text{s.t.}~V_i(t)<\delta,~\forall t\geq t^K,~\forall i \, . \notag
\end{align}
which concludes the proof, because of (\ref{E:cond_asymp}).
}

\bibliographystyle{IEEEtran}
\bibliography{references}

\end{document}